\documentclass[12pt]{amsart}

\usepackage{amsmath}
\usepackage{amscd}
\usepackage{amssymb}
\usepackage{enumerate}
\usepackage{amsfonts}
\usepackage{graphicx}
\usepackage[all]{xy}
\usepackage{mathrsfs}
\usepackage[margin=1in]{geometry}
\usepackage{hyperref}

\usepackage{bbm}
\usepackage{bm}

\newcommand{\fk}{\mathfrak}

\newcommand{\under}{\underline}
\newcommand{\lrarrow}{\longrightarrow}
\newcommand{\rarrow}{\rightarrow}

\newcommand{\xrarr}{\xrightarrow}
\newcommand{\onto}{\twoheadrightarrow}
\newcommand{\into}{\hookrightarrow}

\newcommand{\bul}{\bullet}
\newcommand{\tld}{\tilde}
\newcommand{\wtld}{\widetilde}

\newcommand{\RNum}[1]{\uppercase\expandafter{\romannumeral #1\relax}}
\newcommand{\cid}{\operatorname{cid}}
\newcommand{\di}{\operatorname{dim}}
\DeclareMathOperator{\Tor}{Tor} 
 
\DeclareMathOperator{\Ima}{Im} 
 
\DeclareMathOperator{\Ann}{Ann}

\newtheorem{thm}{Theorem}
\newtheorem{cor}[thm]{Corollary}
\newtheorem{lem}[thm]{Lemma}
\newtheorem{prop}[thm]{Proposition}
\newtheorem{defin}[thm]{Definition}
\newtheorem{conj}[thm]{Conjecture}

\theoremstyle{remark}
\newtheorem{rem}[thm]{Remark}

\newtheorem{disc}[thm]{Discussion}

\begin{document}

\title{ the number of generators of the first Koszul homology of an Artinian ring }
\author{Alex Zhongyi Zhang}
\address{University of Michigan, Ann Arbor, MI 48109, USA}
\email{zhongyiz@umich.edu}

\begin{abstract}
We study the conjecture that if $I,J$ are $\mu-$primary in a regular local ring $(R,\mu)$ with $\di(R)=n$, then $\frac{I \cap J}{IJ} \cong \Tor_1(R/I, R/J)$ needs at least $n$ generators, and a related conjecture about the number of generators of the first Koszul homology module of an Artinian local ring $(A,m)$. In this manuscript, we focus our attention on the complete intersection defect of the Artinian ring and its quotient by the Koszul elements. We prove that the number of generators of the first Koszul homology module of $x_1,...,x_n \in m$ on an Artinian local ring $(A,m)$ is at least $n+ \cid(A)- \cid(\frac{A}{(x_1,...,x_n)A})$, where $\cid A$ denotes the complete intersection defect of the Artinian local ring $A$.  
\end{abstract}

\maketitle
\section{Introduction}
 We study the conjecture that if $(R,m)$ is a regular local ring, $I+J$ is $m$-primary, and $R/I,R/J$ are Cohen-Macaulay, then $\Tor_1(R/I, R/J)$ needs at least $\text{ht}(I) +\text{ht}(J)-n$ generators. The conjecture follows in the equal-characteristic case from the conjecture that the number of generators of $H_1(x_1,...,x_n;A)$ is at least $n$, where $(A,m)$ is an Artinian local ring and $x_1,..,x_n \in m$.

First we recall several concepts in commutative algebra.

\begin{defin}[complete intersection defect]
If $A$ is a local ring and $M$ is a finitely generated $A$-module, let $\nu(M)$ denote the least number of generators of $M$. If $A=T/I$, where $T$ is regular local. Then the complete intersection defect of $A$, denoted by $\cid(A)$, is defined as $\nu(I) -\text{ht}(I)$.
\end{defin}

\begin{rem}
This notion is well defined: see [A].
\end{rem}
\begin{rem}
Notice that if $A$ is a complete intersection, namely a regular local ring modulo a regular sequence, then $\cid(A)=0$. Also, $\cid(A) \geq 0$ for any local ring $A$, since for any ideal $I$, $\text{ht}(I)\leq \nu(I)$ by Krull's height theorem.
\end{rem}

We use $T_{\bul}$ for the total complex of a double or multiple complex.
\begin{defin}[Koszul Complex and Koszul homology]
   Given a ring $A$ and $x_1,...x_n \in A$, for an $A-$module $M$ the Koszul complex of $M$ with respect to $x_1,...,x_n$, denoted by $K_{\bullet}(x_1,...,x_n; M)$, is defined as follows. If $n=1$, $K_{\bullet}(x_1;M)$ is the complex $0\rarrow M \overset{x_1}{\rarrow}M \rarrow 0$, where the left hand copy of M is in degree 1 and the right hand copy in degree $0$. Recursively, for $n>1$, 
   \[    K_{\bullet}(x_1,...x_n; M)=T_{\bullet}(K_{\bullet}(x_1,...x_{n-1}; A)\otimes_{A} K_{\bullet}
      (X_n;M))
   \]
   Said differently, 
   \[K_{\bullet}(x_1,...x_n;M)=T_{\bullet}( K_{\bullet}(x_1;A) \otimes_{A} \ldots 
   \otimes_{A} K_{\bullet}(x_n;A) \otimes_A M)\]
    The $i$th homology of the complex is the $i$th Koszul homology, denoted as $H_{i}(x_1,...,x_n; M)$ or $H_{i}(\under{x};M)$.
   
\end{defin}

 

\paragraph{}
We give several properties of the Koszul complex.

\begin{prop}
 An $A$-linear map $f : M \lrarrow N $ induces, in a covariantly functorial way, a map of Koszul complexes $K_{\bullet}(\under{x};M)\lrarrow K_{\bullet}(\under{x};N)$, and, hence, a map of Koszul homology $H_{\bullet}(\under{x}; M) \lrarrow H_{\bullet}(\under{x};N).$ If $M=N$ and the map is multiplication by $a \in A$, the induced map on Koszul complexes and on their homology is also given by multiplication by the element $a$.
  \end{prop}
  
  This is obvious from the definition of Koszul homology: see Section 1.6 in [BH].
  \begin{cor}
     $I=\Ann_{A}M$ kills all the Koszul homology modules $H_i(\under{x};M)$.
       \end{cor}

   \begin{prop}
   $(x_1,...,x_n)A $ kills every $H_{i}(x_1,...,x_n;M)$
  \end{prop}
 A proof is provided in Section 1.6 in [BH].

   We can characterize the Koszul complex $K(\under{x};A)$ from the point of view of exterior algebra. Let $K_1(\under{x};A)=Au_1 \oplus Au_2\oplus \cdots \oplus Au_n$. Then $K_t(\under{x};A)\cong \wedge^tK_1(\under{x};A)$. Note that the formula for  $d_i$ comes from the fact that it is derivation on the exterior algebra, namely 
   \[ d(u_{j_1}\wedge u_{j_2} \cdots \wedge u_{j_i})=\sum_{t=1}^{i} (-1)^{t-1} x_{j_t} u_{j_1} \wedge u_{j_2} \wedge \cdots \wedge u_{j_[t-1]} \wedge u_{j_{t+1}}\cdots \wedge u_{j_t}
   \]
\begin{rem}
The $d_i$ are uniquely determined by $d_1$, where $d_1(u_j)=x_j$. Given an invertible linear transformation $L: A^n \to A^n$ such that $L(x_i)=y_i$, we have an induced isomorphism $\wedge A^n \to \wedge A^n$, which gives an isomorphism $K_{\bul}(\under{x};A) \cong K_{\bul}(\under{y};A)$. Therefore $H_{i}(\under{x}; A)\cong H_{i}(\under{y};A)$ and $H_{i}(\under{x}; M)\cong H_{i}(\under{y};M)$.
\end{rem}

The following conjecture motivates the study of the minimal number of generators of $H_{1}(\under{x};A)$ when $A$ is an Artinian local ring.
 \begin{conj}
    Let $(R,m)$ be a regular local ring with Krull dimension $n$. Suppose $I+J$ is an $m-$primary ideal in $R$ and $R/I,R/J$ are Cohen-Macaulay. Then $\Tor_1(R/I, R/J)$ needs at least $\text{ht}(I) +\text{ht}(J)-n$ generators. In particular, if $I,J$ are two $m-$primary ideals, then $\Tor_1(R/I, R/J)$ needs at least $n$ generators.
 \end{conj}
 
 \begin{conj}
    If $(A,m)$ is Artin local, $x_1,..,x_n \in m$, then $H_1(\under{x};A)$ has at least $n$ generators as an $A$-module.
 \end{conj}
The fact that the latter conjecture implies the first is one of the reasons we study the minimal number of generators of the first Koszul homology module.

\begin{disc}
We note that Conjecture 10 reduces to the case where the Koszul elements $x_1,..,x_n$ are minimal generators of the ideal they generate. If not, by applying an invertible matrix $N$(see Remark 8) to $x_1,...,x_n$ we may assume $x_{n}=0$. In this case $K_{\bul}(x_1,...,x_{n-1})$ is the homology of the mapping cone of the $0$ map from $K_{\bul}(x_1,...,x_n;M)$ to itself, which implies $K_1(x_1,...,x_{n-1},0;M) \cong K_0(x_1,...,x_{n-1};M) \oplus K_1(x_1,...,x_{n-1};M)$. Applying this to $M=A$ yields the result. 
\end{disc}
\begin{prop}
 Conjecture 10 implies conjecture 9 in the equi-characteristic case.
 \end{prop}
The proof uses Serre's method of reduction to the diagonal: [S], pages V-4 through V-6

  \begin{rem}
  $R,I, J$ as in the proposition above, if $R$ contains a field, we may replace $R$ by $\hat{R}$, and then $\hat{R} \cong k[[X_1,...,X_n]]$, the formal power series ring, and $m=(X_1,..,X_n)$. By Serre's result in reduction to the diagonal,
    $$\Tor_{i}^{R}(\frac{R}{I}, \frac{R}{J}) \cong\Tor_{i}^{S}( S /(I^e+J'^e), R).$$
   Here, $R:= k[[X_1,...,X_n]], R':=k[[Y_1,...,Y_n]]$. Then $R \cong R'$ in an obvious way, where $X_i \mapsto Y_i$, let $J'$ be the image of $J$ in $R'$. $S:= R \widehat{\otimes_{k}} R' \cong k[[X_1,...X_n,Y_1,...,Y_n]]$, $I \subset R$, $J' \subset R'$, and $I^e $ and $J'^e$ are the extensions of $I$ and $J'$ to $S$. 
\end{rem}
We give the proof of Conjecture 9 from Conjecture 10.
\begin{lem}
If $x $ is a regular element on $M$, then $H_i(x,y_1,...,y_n;M) \cong H_i(y_1,..,y_n; \frac{M}{xM}).$
\end{lem}
A proof can be found in Section 1.6 in [BH].
\begin{proof}[Proof of Conjecture 9 from Conjecture 10]

  Recall that $\nu(M)$ denotes the least number of generators of a module $M$. First enlarge $k$ to an infinite field by replacing $(R,m)$ by $R[t]_{m R[t]}$. The hypothesis and conclusion are not affected. We have 
   $$\Tor_1^R(\frac{R}{I},\frac{R}{J})\cong \Tor_1^S(\frac{R}{I} \hat{\otimes}_k \frac{R'}{J}; \frac{S}{X_1-Y_1,...,X_n-Y_n})$$
     Let $(B,\mu)=\frac{R}{I} \hat{\otimes}_k \frac{R'}{J'}$. Then by the results of [S], $B$ is a Cohen-Macaulay ring of dimension $\delta=\di(\frac{R}{I}) + \di(\frac{R'}{J'})=n-\text{ht}(I)+n-\text{ht}(J)$. The images of $X_i- Y_i$ generate a $\mu$-primary ideal in $B$. Since 
   $$\frac{R}{I} \hat{\otimes}_k \frac{R'}{J'} \otimes  \frac{S}{(X_1-Y_1,...,X_n-Y_n)} \cong \frac{R}{I} \otimes_R \frac{R}{J} \cong \frac{R}{I+J},$$
   and $I+J$ is $m-$primary. 
   Let $Z_i=X_i-Y_i$.
    Because $Z_1,...,Z_n$ is a regular sequence in $S$, the Koszul complex $K_{\bullet}(\under{Z};S)$ gives a resolution of $S/(\under{Z})$. Tensoring this resolution with $B$ gives $\Tor_{i}^{S}(B , S/(\under{Z}))$, but it also gives the $H_{i}^{S}(\under{Z}; B)$
    
Now that $k$ is infinite, we can choose $\delta$ $ k$-linear combinations of $X_i-Y_i$ that form a regular sequence on $B$, say $z_1,..,z_{\delta}$, and let $z_1,...,z_n$ extend this to elements that span the same $k$-vector space as the $X_i-Y_i$.
   $$H_1(X_1-Y_1,...,X_n-Y_n;B)\\
   = H_1(\under{Z};B)
   \cong H_1(z_1,..,z_n;B)\\
   \cong H_1(z_{\delta+1},...,z_n; \frac{B}{z_1,...,z_{\delta}})$$ 
   where the last congruence is a result of iteration of Lemma 13.
   But $\frac{B}{z_1,...,z_{\delta}}$ is Artinian local. Therefore, by Conjecture 10, it will need $n-\delta = n-(n-\text{ht}(I)+n-\text{ht}(J)) = \text{ht}(I) +\text{ht}(J)- n $ generators. 
\end{proof}
We need the following well-known fact.
\begin{lem}
Given a finite complex $A_{\bullet}: 0 \to A_{k} \xrarr{d_k} \cdots \xrarr{d_1} A_0 \to 0$ such that each $A_i$ has finite length, the alternating sum of the lengths of homologies equals 0, i.e. $\sum_{i=0}^{k} (-1)^{i} l(H_i(A_{\bullet}))=0$
\end{lem}

\begin{thm}
Suppose $(A,m)$ is local, $x_1,...,x_n \in m$, $M$ is finitely generated, and $\frac{M}{(x_1,...,x_n)M} $ has finite length. If $\di(M) <n$, then $\chi(x_1,,.x_n;M):=\sum(-1)^i l(H_i(\under{x};M))=0$. If $\di(M)=n$, $\chi(\under{x};M) $ is the multiplicity of M on $(x_1,...,x_n)R$.
\end{thm}
A proof can be found in [S].Th\'eor\`eme 1, page $\RNum{5}$-12

For Conjecture 10, the cases when $n=1$ or $2$ are known. We include the argument for the convenience of the reader.
\begin{proof}[Proof of Conjecture 10 when $n=1$ or $2$]
 When n=1, $H_1(x;A)$is just $\Ann_{A}x$, since $x \in m$ and $A$ is Artinian, so $m$ consists of zero divisors, thus, $\Ann_{A}x \neq 0$, so it has to be generated by at least 1 element as an $A$-module.\\
 When n=2, suppose $\nu(H_1(\under{x};A))<2$, then it is cyclic. Since $A$ is Artinian, every finite module has finite length, and the Euler characteristic of the Koszul complex is 0. Thus, by Theorem 15,  $l(H_1(x_1,x_2; A))-l(\frac{A} {(x_1,x_2)})- l(\Ann_A(x_1,x_2))=0$. We know $\Ann_A(x_1,x_2) \neq 0$, so $\nu (H_1)\neq 0$, and if $H_1(\under{x};A)$ is cyclic, then since it is killed by $(x_1,x_2)$ by Proposition 7, it is a quotient module of $A/(x_1,x_2)$. Therefore, $l(H_1(\under{x};A)) \leq l(A/(x_1,x_2))$. But, as we noted earlier, $l(\Ann_A(x_1,x_2)) >0$, so the alternating sum $l(H_1(x_1,x_2; A))-l(\frac{A} {(x_1,x_2)})- l(\Ann_A(x_1,x_2)) <0$, a contradiction. Thus, we have $\nu (H_1(x_1,x_2;A)) \geq 2$.
 
\end{proof}

\section{ main result}

\begin{thm}{(Main Theorem)}
If $(A,m,k)$ is an equicharacteristic Artinian local ring, and  $x_1,...,x_n \in m$, then $\nu( H_1(\under{x};A)) \geq n+\cid(A)-\cid(\frac{A}{(\under{x)}})$.
\end{thm}

Fix a coefficient field $k \subset A$. Suppose $y_1,...,y_m$ is a minimal set of generators of the maximal ideal of $\frac{A}{(\under{x})A}$. We may write $A/(\under{x})=\frac{k[[Y_1,...,Y_m]]}{I}$. Find lifts $\tilde{y_1},...,\tilde{y_m} \in A$ of $y_1,..,y_m$. Let $(T,\fk{n})$ denote $k[[X_1,...,X_n, Y_1,..,Y_m]]$. Let $\theta :K[[Y_1,...,Y_m]] \onto \bar{A}=\frac{A}{(\under{x})A},$ such that $ Y_i \mapsto y_i $ and $\tld{\theta} : T \onto A,  X_j \mapsto x_j, Y_i \mapsto \tilde{y_i}$ be the lifts of $\theta$ corresponding to the lift from $y_i$ to $\tilde{y_i}$. 

Let $\fk{a}$ denote $\ker(\tilde{\theta})$ and $I$ denote $\ker(\theta)$. Then we have a map $\pi: k[[X_1,...,X_n,Y_1,...,Y_m]] \to k[[Y_1,...,Y_m]] $, where $X_i \mapsto 0, Y_i \mapsto y_i$. Therefore $\pi(\fk{a}) \subset I$. Let $\phi: \fk{a} \to I$ denote the restriction of $\pi$ to $\fk{a}$.We claim $\phi$ is surjective because for any $F \in I $, if we abuse notation and write $F$ as its image in the obvious injective map from $k[[\under{Y}]]$ to $T=k[[\under{X},\under{Y}]]$, we shall have $\theta(F)=0$ $\Rightarrow  \tilde{\theta}(F) \in (x_1,...,x_n)A $ $\Rightarrow \tilde{\theta}(F)=\sum_{i=1}^{n} x_i a_i$ for certain $a_i \in A$. Lift the $ a_i $ to some $G_i \in T$, then $\tilde{F_j}=F_j -\sum X_i G_i \in \fk{a}$.
Therefore, we have the following commutative diagram:
\[
 \xymatrix{
 0 \ar[r] &\fk{a} \ar[r]  \ar@{->>}[d]^{\phi} &k[[X_1,...,X_n,Y_1,...,Y_m]]  \ar@{->>}[d]^{\pi} \ar@{->>}[r]^-{\tld{\theta}}   &A \ar@{->>}[d] \ar[r] &0 \\
0 \ar[r] &I \ar[r]               		 &k[[ Y_1,...,Y_m]]    \ar@{->>}[r]^-\theta 		& A/(\under{x})A   \ar[r] &0 }
\]
 Note that the rows are exact.

Before we prove the theorem, we shall prove several lemmas.
 We want to find polynomials $F_1,...,F_m$ such that they form a part of a minimal set of generators of $I=\ker(\theta)$ and they are also a system of parameters in $k[[Y_1,...,Y_m]]$. 
By a standard prime avoidance argument, we have the following lemma.
\begin{lem}
There exist elements $F_1,...,F_m$ such that they form a part of a minimal set of generators of $I$ and they are also a system of parameters in $k[[Y_1,...,Y_m]]$. 
\end{lem}

\begin{proof}
     Suppose $g_1,..,g_i$ are minimal generators of $I$ that are part of a system of parameters ($i=0$ is allowed). Let $p_1,...,p_k$ be the minimal primes of $(g_1,...,g_i).$ Then $$I \nsubseteq \bigl( (g_1,..,g_n)+(Y_1,...Y_m)I \bigr) \cup (\cup_j p_j),$$ unless $i=m$. Pick $g_{i+1} \in  I- \bigl( (g_1,..,g_n)+(Y_1,...Y_m)I )\cup (\cup_j p_j)\bigr)$, then $g_1,...,g_{i+1}$ are part of a system of parameters of $K[[Y_1,...,Y_m]] $ and part of a minimal set of generators of $I$.
\end{proof}

Since $\phi: \fk{a} \onto I$, the $F_j$ have lifts in $T$ that are in $\fk{a}$. Say $\tilde{F_j}=F_j - \sum X_i G_{ji}$, where $G_{ji} \in T$. Let $S=T/(\tilde{F_1},...,\tilde{F_m})$. We want to show that $\wtld{F_j}$ is also part of a minimal set of generators for $\fk{a}$.

\begin{lem}
If $F_1,...,F_m$ is part of a minimal set of generators for $I=\ker(\theta)$, then their lifts $\wtld{F_j} \in \fk{a} $ are also part of a minimal set of generators for $\fk{a}=\ker(\tld{\theta})$. 
\end{lem}

\begin{proof}
Since $\wtld{F_j}$ are already in $\fk{a}$, it suffices to prove that their images are linearly independent over $k$ in $\fk{a}/\fk{n}\fk{a}$.
   
  Consider $\fk{a}/\fk{na} \into T/\fk{na} \onto T/(\fk{na}+(\under{X})) \cong k[[\under{Y}]] / \Ima( \fk{na}) \cong k[[\under{Y}]]/ (\under{Y})I$, where the last isomorphism holds because $\Ima(\fk{n})=(\under{Y})$ and $\Ima(\fk{a})= I$. The images $\gamma_j$ of the $\tilde{F_j}$ in $\fk{a/na}$ maps to elements $\bar{\gamma_j}$ in $I/(\under{Y})I$, which are the same as the images of the $F_j$. The $\bar{\gamma_j}$ are $k$-linearly independent, so the $\gamma_j$ are $k$-linearly independent.
\end{proof}

Recall that $S=T/(\tilde{F_1},...,\tilde{F_m})$. $A$ is a quotient of $S$. Consider the minimal free resolution of $A$ over $S$:

\[
  \cdots \to S^{n_2} \to S^{n_1} \to S \onto A
\]

\begin{lem}
If $x_1,...,x_n $ is a regular sequence in a ring $R$, then $K_{\bul}(x_1,...,x_n;R)$ gives a free resolution of $R/(x_1,...,x_n)R$. Furthermore, if $(R,m)$ is local and $(x_1,...,x_n) \in m$, then this is a minimal free resolution.
\end{lem}
The proof is given in section 1.6 in [BH].

\begin{proof}[Proof of the Main Result]
 The $x_i$ form a regular sequence on $S$ (since $x_i,\tilde{F_j}$ form a regular sequence on $T$.) Thus $H_1^A(\under{x};A) \cong H_1^{S}(K_{\bul } (\under{x};S) \otimes A) \cong \Tor_1^S( \frac{S}{(\under{x})S};A)$.

Tensor the minimal resolution of $A$ with $S/(\under{x}S)$. This yields the following complex:
\begin{equation}
       \cdots \to \Bigl(\frac{S}{(\under{x})}\Bigr)^{n_2} \to \Bigl(\frac{S}{(\under{x})}\Bigr)^{n_1} \to \frac{S}{(\under{x})}\to \frac{A}{(\under{x})}.  \label{eq1}
\end{equation}

The matrices of the maps of this resolution have entries that are contained in $(y_1,...,y_m)$, since the original resolution of $A$ as an $S-$module was a minimal one over $S$, so that the maps of the original resolution were contained in $(x_1,...,x_n,y_1,..., y_m)=m_S$.

 Now $$A=T/\fk{a}=T/(\wtld{F_1},...,\wtld{F_m}, G_1,...,G_{t})=S/J$$
 where
 $$\wtld{F_1},...,\wtld{F_m}, G_1,...,G_{t}$$
  is a minimal generating set of $\fk{a}$ and $J=(\overline{G_1},...,\overline{G_{t}})$,i.e, ideal generated by images of $G_k$ in $S$. Then $t$ is the minimal number of generators of $J=\ker(S \onto A)$. Thus $t=n_1$. 

Now $\fk{a}$ is minimally generated by $m+n_1$ elements. Hence $m+n_1 = \nu(\fk{a}) = \cid(A)+ \di(T) =\cid(A)+(m+n)$. thus $n_1 = \cid(A)+ n$.

Let $\bar{S}$ denote $\frac{S}{(\under{x})}$, $\bar{A}$ denote $\frac{A}{(\under{x})}$. $W$ denote $\ker(\bar{S} \onto \bar{A})$, $Z$ denote $\ker(\bar{S}^{n_1} \to \bar{S})$, $B$ denote $\Ima(\bar{S}^{n_2}) \in m_{\bar{}S} \cdot \bar{S}^{n_1}.$ Then by definition, $\nu(W)= \cid(\bar{A})$.
We have the following short exact sequences:
\begin{equation}
	0 \to Z \to \bar{S}^{n_1} \to W \to 0  \text{			} \label{eq2}
\end{equation}
 and 
 \[
 0 \to B \to Z \to H \to 0.
 \]
 Where $H= Z/B$, which is the first homology module of $(\ref{eq1})$, i.e. $\Tor_1^{S}(\bar{S};A)=H_1^A(\under{x};A)$. 
 Now tensor (\ref{eq2}) with  $K$, we get the exact sequence:
 \[
    Z\otimes K \to K^{n_1} \to  W\otimes K \to 0.  
 \]
 But $B \in m \cdot \bar{S}^{n_1}$, so $\frac{Z}{B} \otimes K \cong Z\otimes K$. Hence,
 \[
H\otimes K \to K^{n_1} \to  W\otimes K \to 0.  
\]
 Therefore, $n_1=\nu(K^{n_1}) \geq \nu(H) +\nu(W)$, and $\nu(H)\geq n_1-\nu(W)=n+\cid(A)-\cid(\bar{A})$.
 
 \end{proof}
 
\begin{cor}
When $\cid(A) \geq \cid(\frac{A}{(x_1,...,x_n)A})$, we have $\nu(H_1(x_1,..,x_n;A)) \geq n$. 
\end{cor}

\section{acknowledgements}
I would like to thank the University of Michigan and the Department of Mathematics. I would also like to thank my advisor Melvin Hochster for his advice and guidance, without which these results would not have been obtained.

\end{document}